\documentclass[psamsfonts]{amsart}

\usepackage{amssymb,amsfonts}
\usepackage[all,arc]{xy}
\usepackage{enumerate}
\usepackage{mathrsfs}
\usepackage{bm}
\usepackage{pgfplots}
\usepackage[utf8]{inputenc} 
\usepackage[T1]{fontenc}    
\usepackage{hyperref}       
\usepackage{url}            
\usepackage{booktabs}       
\usepackage{amsfonts}       
\usepackage{nicefrac}       
\usepackage{microtype}      
\usepackage{lipsum}
\usepackage{mathtools}
\usepackage{amsmath}
\usepackage{amsthm}
\usepackage{amsfonts}
\usepackage{amssymb}
\usepackage{tikz-cd}
\usepackage{tikz}
\usepackage[T1]{fontenc}
\usepackage{lmodern}
\hypersetup{
pdftitle={On the moduli space of ricci flat metrics on a k3 surface},
pdfsubject={Mathematics, Differential Geometry, K3 surfaces, Moduli Spaces},
pdfauthor={Degen David},
pdfkeywords={bi-quotient, Ricci-flat, Einstein, Moduli, Fundamental group, Topology, Differential geometry, arithmetic group, Calabi-Yau, K3 surface, Complex Geometry, Riemannian Geometry, Period Map, Torelli Theorem for K3 surfaces, arXiv, Preprint, Mathematics}
}

\theoremstyle{definition}

\theoremstyle{remark}

\newtheorem{theorem}{Theorem}[section]

\newtheorem{proposition}[theorem]{Proposition}

\theoremstyle{definition}

\renewcommand{\r}{\text{Ric}}

\makeatletter
\let\c@equation\c@thm
\makeatother
\numberwithin{equation}{section}

\title[On the Moduli Space of Ricci flat metrics on a $K3$ surface]{On the Topology of the Moduli Space of Ricci flat Metrics on a $K3$ surface}
\author{David Degen}
\address{Karlsruher Institut fÜr Technologie, 
Institut fÜr Algebra und Geometrie,
Arbeitsgruppe Differentialgeometrie 
Englerstr.  2, D-76131 Karlsruhe, Germany}
\email{david.degen@kit.edu}
\date{ \today\\
\textit{E-mail}: \tt{david.degen@kit.edu}}
\begin{document}
\begin{abstract}
We show that the moduli space of Ricci flat metrics of unit volume (including orbifold metrics) on a K3 surface is simply connected and that it has the same rational cohomology as the automorphism group of the K3 lattice $(-E_8)^{\oplus 2}\oplus U^{\oplus 3}$.
\end{abstract}
\maketitle
\section{Introduction}
Ricci flat metrics in dimension 4 have been studied for more than 100 years, prominently from their appearance in General Relativity but also their moduli spaces in String Theory and Mirror Symmetry. However many important problems remain unsolved, like finding all closed $4$-manifolds admitting a Ricci flat metric. The only known examples of smooth Ricci flat Riemannian metrics on such manifolds are finitely covered by a 4-torus $T^4$ or a $K3$-surface. On the other hand, in both cases one has a rather good description to the corresponding  moduli space of Ricci flat metrics $\mathcal{M}_{Ric=0}(M)$. These moduli spaces can be described as a bi-quotient
\[
	\Gamma \backslash L / K,
\]
where $\Gamma$ is a Lattice in the Lie group $L$ and $K$ a maximal compact subgroup of $L$. A natural question is then to ask for the topology of these spaces. In case of $M$ being a $4$ torus one finds that any Ricci flat metric is in fact flat, so that  $\mathcal{M}_{Ric=0}(T^4)$ coincides with the moduli space of flat structures on $M$. This space is known to be homeomorphic to
\[
	GL(4,\mathbb{Z})\backslash GL(4)/O(4)
\]
and by results of Tuschmann and Wiemeler on the moduli spaces of flat structures on tori \cite{Tuschmann.} one finds that $\mathcal{M}_{Ric=0}(T^4)$ is simply connected and that its third rational homotopy group is isomorphic to $\mathbb{Q}$. In this note we show that a similar result holds for K3 surfaces. Here the moduli space can be identified with 
\[
 SO(\Lambda_{K3}) \backslash O(3,19)/ SO(3)\times O(19).
\]
However, for this moduli space we allow the metrics to have orbifold type singularities. Precisely we show the following theorem.
\subsection*{Main Theorem}
The moduli space of Ricci flat metrics of unit volume, including orbifold metrics, on a $K3$ surface is simply connected. 
\\
\\
Furhter, we show that this space has the same rational cohomology as $O(\Lambda_{K3})$, the automorphism group of $H^2(X,\mathbb{Z})$ of a $K3$ surface $X$ with its cup-pairing.
\\
This note is organized as follows: In the first section we provide some background on $K3$ surfaces, which for example can be found in \cite{Huybrechts.2016}. Further, we describe the moduli space of Ricci flat metrics as a Grassmann space of positive oriented $3$-planes, modulo the automorphism group of the $K3$ lattice. 
In the second section we give the proof of our main theorem, by showing that each element of a generating set of $O(\Lambda_{K3})$ has a fixed point in $L/K$ and that the quotient of such an action must be simply connected.
\section{Preliminaries}
A $K3$ surface $X$ can be defined as a simply connected compact complex surface with vanishing first Chern class.
Prominent examples are given by the Fermat surface, defined by the zero locus of  
\[
	z_1 ^4+z_2 ^4+z_3 ^4+ z_4 ^4
\]
in $\mathbb{C}P^3$ or more general as any hypersurface of degree $4$ in $\mathbb{C}P^3$. Other examples are given by blowing up the singularities in an involuted $4$-torus. More precisely, let $\Lambda$ be a lattice in $\mathbb{C}^2$ and let $T^4$ be the torus $\mathbb{C}^2 / \Lambda$. Then on this torus the group $\mathbb{Z}_2$ is acting by taking $(z_1,z_2)$ to $(-z_1,-z_2)$ and the resulting orbit space, $T^4/\mathbb{Z}_2$, is a complex orbifold having $16$ singular points. Taking the blow up of each singularity results in having a $K3$ surface, called Kummer surface. It is interesting to note that the Fermat surface is in fact a Kummer surface. However, not every Kummer surface is a hypersurface, they might not even be algebraic and in general there are an abundance of different types of $K3$ surfaces. On the other hand it is a well known fact that when one forgets about the complex structure, and only views a $K3$ surface as a differentiable manifold, then they are all the same, i.e.\ any two $K3$ surfaces are diffeomorphic. 
\\
\\
Let $M$ be the underlying differentiable manifold of a $K3$ surface.
One very important topological invariant for studying $K3$ surfaces, i.e.\ complex structures on $M$, is given by the second cohomology group $H^2(M,\mathbb{Z})$ together with the cup-product $(\cdot,\cdot)$, or dually, by homology with its intersection pairing.
This defines a lattice which is even and unimodular of signature $(3,19)$. Since such lattices are unique, the lattice $H^2(M,\mathbb{Z})$ is isomorphic to the lattice with intersection matrix given by 
\[
(-E_8)\oplus (-E_8)\oplus \begin{pmatrix}
0 & 1 \\
1 & 0
\end{pmatrix} \oplus \begin{pmatrix}
0 & 1 \\
1 & 0
\end{pmatrix} \oplus \begin{pmatrix}
0 & 1 \\
1 & 0
\end{pmatrix},
\] 
where the matrix $E_8$ is associated to the Dynkin diagram
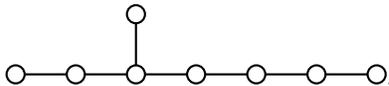
\begin{center}
  \begin{tikzpicture}[scale=.4]
    \draw (-1,1) node[anchor=east]  {};
    \foreach \x in {0,...,6}
    \draw[thick,xshift=\x cm] (\x cm,0) circle (3 mm);
    \foreach \y in {0,...,5}
    \draw[thick,xshift=\y cm] (\y cm,0) ++(.3 cm, 0) -- +(14 mm,0);
    \draw[thick] (4 cm,2 cm) circle (3 mm) ;
    \draw[thick] (4 cm, 3mm) -- +(0, 1.4 cm);
  \end{tikzpicture}.
\end{center}
That is, to each node we associate a basis vector $e_i$ and the intersection matrix $E_8$ is given by $(e_i,e_j)=2$ if $i=j$, $(e_i,e_j)=-1$ if $e_i$ and $e_j$ are connected by an edge, and $(e_i,e_j)=0$ otherwise.
We will refer to this lattice as the $K3$ lattice and denote it by $\Lambda_{K3}$. Further, note that $H^2(M;\mathbb{R})$ is isomorphic to $\Lambda_{K3}\otimes \mathbb{R}$, which defines a pseudo euclidean space $(\mathbb{R}^{22},(\cdot ,\cdot))$, where $(\cdot,\cdot)$ is a bilinear pairing of signature $(3,19)$. We will also  denote this by $\mathbb{R}^{3,19}$.
\\
\\
Define the automorphism group $O(\Lambda_{K3})$ of the $K3$ lattice as the group of isomorphisms $g\colon \Lambda_{K3}\to \Lambda_{K3}$ such that $(gx.gy)=(x.y)$ for all $x,y\in \Lambda_{K3}$. Notice that $O(\Lambda_{K3})$ is a discrete subgroup of the Lie group \[
O(3,19)=\{ g \in GL_\mathbb{R}(22)\  \vert \ (gx,gy)=(x,y) \text{ for all } x,y \in \mathbb{R}^{22}\}.
\]
The group $O(3,19)$ naturally acts on $\mathbb{R}^{3,19}$ and has $4$ connected components corresponding to whether orientation on the positive definite subspaces are maintained or being reversed. Let $SO(\Lambda_{K3}):=O(\Lambda_{K3})\cap SO(3,19;\mathbb{R})$, and let $G^{3}_{o,+}(\mathbb{R}^{3,19})$ denote the Grassman space of oriented and positive $3$-planes in $\mathbb{R}^{3,19}$. \\
\\
By our definition of $K3$ surfaces, plus the fact that any $K3$ surface is Kähler \cite{Siu.1983}, it follows from the famous Calabi-Yau theorem \cite{Yau.1978} that any $K3$ surface possesses a Ricci flat Kähler metric. 
On the other hand, if $M$ is the underlying smooth manifold of a $K3$ surface, then any Ricci flat metric $g$ on $M$ will determine a whole two-sphere of complex structures, each giving $M$ the structure of a $K3$ surface and making $g$ Kähler, i.e.\ $g$ determines a hyperkähler structure on $M$. Further, any Einstein metric on $M$ will turn out to be Ricci-flat and the terms Einstein metric, Calabi-Yau metric and hyperkähler metric can be used interchangeably, see \cite{Besse.2008}. It is that interplay between complex structures and Ricci flat metrics which makes it possible to study the moduli space $\mathcal{M}_{\r=0}(M)$ by means of the classical Torelli theorem for $K3$-surfaces, which states that two $K3$-surfaces $X$, $X'$ are isomorphic if and only if there is a Hodge isometry between $H^2(X;\mathbb{Z})$ and $H^2(X';\mathbb{Z})$.
In the following we will define the refined period map, which gives a good picture of the moduli space $\mathcal{M}_{\r=0}(M)$ by studying Hodge structures relative to the $K3$ lattice. However, we will only present a very small picture of this map and refer to some of the original works, \cite{PjateckiuSapiro.1971},\cite{Burns.1975},\cite{Looijenga.1981},\cite{Todorov.1980},\cite{Morrison.1983},\cite{Looijenga.1981b},\cite{Kobayashi.1987} and the expository paper \cite{Kobayashi.1990} for more details on this topic. 
\\
\\
We start by defining $\mathcal{T}_{\r=0}(M)$ to be the Teichmüller space of Ricci flat metrics of unit volume on $M$. That is, $\mathcal{T}_{\r=0}(M)$ is the quotient of the space of unit volume Ricci flat metrics on $M$ by $\text{Diff}^{\ id}(M)$, the group of diffeomorphisms which induce the identity on $H^2(M;\mathbb{Z})$. Then the refined period map $\mathcal{P}$ is defined by mapping a Ricci flat metric $g$ to the space of harmonic and self-dual two forms $\mathcal{H}^+(g)$ with respect to the metric $g$. Notice that $\mathcal{H}^+(g)$ is an oriented positive $3$-plane in $H^2_{DR}(M)$ so that the refined period map is given by
\begin{align*}
\mathcal{P}\colon \mathcal{T}_{\r=0}(M) &\to G_{o,+}^3(H^2(M,\mathbb{R}))
\\
[g] & \mapsto \mathcal{H}^+(g).
\end{align*}
Using the Torelli theorems for $K3$-surfaces it was shown that $\mathcal{P}$ is a homeomorphism onto an open and dense subspace $T\subset G^3_{o,+}$. This subspace is precisely given by those $3$-planes $\tau$ in $H^2(M,\mathbb{R})$ whose orthogonal complement $\tau^\bot$ does not intersect the root system of integral $(-2)$-classes given by $\Gamma=\{ \gamma \in H^2(M;\mathbb{Z})\  \vert \  (\gamma,\gamma)=-2)\}$, i.e.\ 
\[T=\{\tau \in G_{o,+}^3 \vert \tau^\bot \cap \Gamma=\emptyset\}.
\]
Thus one can think of the Teichmüller space $\mathcal{T}_{\r=0}(M)$ as a good space having holes of codimension $3$. To obtain the moduli space $\mathcal{M}_{\r=0}(M)$, note that two Einstein metrics $g_1,g_2$ on $M$ are isometric if there is an orientation preserving diffeomorphism $\phi \in \text{Diff}^{\ o}(M)$ so that $g_1$ is the pull back metric $\phi^*g_2$. Hence we get $\mathcal{M}_{\r=0}(M)$ if we divide out the action of $\text{Diff}^{\ o}(X)/\text{Diff}^{\ id}(X)$ on $\mathcal{T}_{\r=0}(M)$. From \cite{Barth.2004} we know that the group $\text{Diff}^{\ o}(M)/\text{Diff}^{\ id}(M)$ is isomorphic to $SO(\Lambda_{K3})$. Consequently we obtain that the moduli space of smooth Ricci-flat metrics is an open and dense subset of $SO(\Lambda_{K3})\backslash G_{o,+}^{3}$. Compare with \cite{Besse.2008}.
\\
\\
An obvious question now arising is, whether the holes in the moduli space, or the Teichmüller space, can be filled in a natural way. This was answered by Kobayashi and Todorov in \cite{Kobayashi.1987}, they showed that points in the complement of $T$ correspond to Einstein-orbifold-metrics on generalized $K3$-surfaces. A generalized $K3$-surface is a compact complex surface having at most rational double points and whose minimal resolution is a $K3$ surface and they showed that the moduli space including those metrics is isomorphic to 
\[
	SO(\Lambda_{K3})\backslash G_{o,+}^{3}.
\]
Notice that rationl double points in $dim=2$ are of the form $\mathbb{C}^2/G$, where $G$ is a finite subgroup of $SU(2)$. In particular these are orbifold singularities. \\
Later Anderson showed in \cite{Anderson.1992}, with more geometrical techniques, that this result can be obtained by studying converging sequences of Einstein metrics with respect to the $L^2$-metric or the Gromov-Hausdorff metric. In fact, he showed that the period map $\mathcal{P}$ extends to an isometry between the closure of $\mathcal{M}_{\r=0}(M)$ with respect to the intrinsic $L^2$-metric and the metric on $SO(\Lambda_{K3})\backslash G_{o,+}^{3}$ induced by the natural metric on $O(3,19)/SO(3)\times O(19)=G_{o,+}^{3}$. Thus we denote by $\overline{\mathcal{E}_{K3}}$ the moduli space of unit volume Einstein orbifold-metrics. Note that from \cite{Allan.1966} we know that $SO(\Lambda_{K3})$ acts properly and discontinuously on $G_{o,+}^3$ so that $\overline{\mathcal{E}_{K3}}$ is a $57$-dimensional orbifold. Further, note that if we drop the restriction of unit volume metrics we obtain an additional factor $\mathbb{R}^+$, so that the moduli space of Einstein orbifold metrics is given by
\[
	 SO(\Lambda_{K3}) \backslash O(3,19)/ SO(3)\times O(19) \times \mathbb{R}^+.
\]

\section{Topological Properties}
We show that $\overline{\mathcal{E}_{K3}}$ is simply connected and that $H^k(\overline{\mathcal{E}_{K3}},\mathbb{Q})\cong H^k(O(\Lambda_{K3}),\mathbb{Q})$.
\\
\\
Note that the Grassmann space of positive 3-planes in $G^+_{3}(\mathbb{R}^{3,19})$ can be written as the symmetric space $O(3,19)/O(3)\times O(19)$. Here we drop the condition on orientation.
\begin{proposition}
The moduli space of Ricci-flat metrics (including orbifold metrics) on a $K3$ surface is homeomorphic to $O(\Lambda_{K3})\backslash G^+_{3}(\mathbb{R}^{3,19})$. 
\end{proposition}
\begin{proof}
From the discussion above we know that we can identify the moduli space of Ricci-flat metrics (including orbifold metrics) $\overline{\mathcal{E}_{K3}}$ with the orbit space 
\[
SO(\Lambda_{K3})\backslash G^+_{o,3}(\mathbb{R}^{3,19}).
\]
Since $O(3,19)$ acts transitively on $G^{+}_{o,3}(\mathbb{R}^{3,19})$ with isotropy group $SO(3)\times O(19)$ we can identify the Grassman space $G^+_{o,3}(\mathbb{R}^{3,19})$ with the symmetric space 
\[
O(3,19)/SO(3)\times O(19).
\]
Therefore $\overline{\mathcal{E}_{K3}}$ coincides with the bi-quotient
\[
SO(\Lambda_{K3})\backslash O(3,19)/SO(3)\times O(19)
\]
and  as both $-id \in SO(\Lambda_{K3})$ and $(id,-id)\in SO(3)\times O(19)$ lie in the center of $O(3,19)$ we can in fact identify $\overline{\mathcal{E}_{K3}}$ with
\[
O(\Lambda_{K3})\backslash O(3,19)/O(3)\times O(19).
\]
The symmetric space $O(3,19)/O(3)\times O(19)$ on the other hand is just $G^+_{3}(\mathbb{R}^{3,19})$ the Grassmann space of positive $3$-planes in $\mathbb{R}^{3,19}$ where we forget about orientation.
\end{proof}
From now on let $\Gamma\coloneqq O(\Lambda_{K3})$ and $X\coloneqq O(3,19)/O(3)\times O(19)$. 
\begin{proposition}
The projection $\pi\colon X\to X/\Gamma$ has the path lifting property.
\end{proposition}
\begin{proof}
First notice that when we can lift a path locally we can lift it globally as $I$ is compact.
From \cite{Palais.1961} we know that the group $\Gamma$ has the slice property, i.e.\ that any point $x\in X$ has a slice $S$ such that a neighborhood of $[x]$ in the orbit space is homeomorphic to $S/\Gamma_x$. Since $\Gamma_x$ is finite, so in particular a compact Lie Group we can lift the path locally to $S$ by \cite[Chapter II Theorem 6.2]{Bredon.1972} and hence globally. 
\end{proof}
\begin{theorem}
The moduli space of Ricci-flat metrics (including orbifold metrics) on a $K3$ surface is simply connected.
\end{theorem}
\begin{proof}
From \cite{.1964} we know that the group $O(\Lambda_{K3})$ is generated by reflections 
\[
s_\delta\colon x \mapsto x + (x.\delta)\delta
\]
with $\delta$ satisfying $(\delta.\delta)=\pm2$. Since for $(\delta.\delta)=-2$ there is a positive $3$-plane contained in the reflection hyperplane $\delta ^\bot$ and for the other case $(\delta.\delta)=2$, there is a $2$-dimensional positive subspace $V'\subset \delta ^\bot$ so that $V \coloneqq V'\oplus \delta\mathbb{R}$ is fixed by $s_\delta$. Now for every such reflection $s_\delta$ there is some $V\in G^+_{3}(\mathbb{R}^{3,19})$ such that $s_\delta(V)=V$. Thus there is a generating set of $O(\Lambda_{K3})$ acting properly and discontinuously on a contractible space for which each element has a fixed point. Now, basically following \cite{Armstrong.1982}, we show that the quotient space of such an action is simply connected. \\
We define a surjective group homomorphism 
\[
\pi\colon \Gamma \to \pi_1(X/\Gamma).
\]
by fixing some $x_0$ in $X$ and for any $\gamma \in \Gamma$ we choose a path $\alpha$ connecting $x_0$ and $\gamma x_0$. On the quotient the path $\alpha$ will then define a loop and it is not hard to see that this definition is independent of the choice of $\alpha$ and $x_0$. To see surjectivity, we know from the above proposition that the quotient map $X\to X/\Gamma$ has the path lifting property, so that any loop $\bar \alpha$ lifts to a path $\alpha$ connecting $x_0$ with $\gamma_\alpha x_0$ for some $\gamma_\alpha \in \Gamma$. Let $\gamma$ be an element in $\Gamma$ which has a fixed point $x$ in $X$. Then because of independencies of the choices for $\pi$ we can choose $x=x_0$ and the constant path $\alpha_\gamma$ connecting $x$ with $\gamma x$, so that $[\alpha_\gamma]=1$ in the fundamental group. By the discussion above we know that $\Gamma$ is generated by such elements and thus $Ker(\pi)=\Gamma$ and hence $\pi_1(X/\Gamma)=\{1\}$. 
\end{proof}

For the cohomology of the space $\overline{\mathcal{E}_{K3}}$ we can show that it coincides, at least rationally, with that of an arithmetic subgroup of $O(3,19)$, i.e.\ 
\begin{theorem}
The rational cohomology of $\overline{\mathcal{E}_{K3}}$ coincides with the rational group cohomology of $O(\Lambda_{K3})$.
\end{theorem}
\begin{proof}
The space $X=O(3,19)/O(3)\times O(19)$ is a model for the classifying space of proper $\Gamma=O(\Lambda_{K3})$ actions $E_\mathcal{FIN}\Gamma$, 
see \cite[Theorem 4.4]{Luck.2005}. It further follows that there is, up to $\Gamma$-homotopy, a unique $\Gamma$-map
\[
 E_\mathcal{TR}\Gamma \to E_{FIN}\Gamma,
\]
where $E_\mathcal{TR}\Gamma=E\Gamma$ is the universal cover of $B\Gamma$, the classifying space of the group $\Gamma$. 
By \cite[Lemma 4.14]{Luck.2007} the induced map on the orbit spaces, 
\[
	\Gamma \backslash E\Gamma \to \Gamma \backslash E_\mathcal{FIN}\Gamma,
\]
yields an isomorphism in rational cohomology. From the discussion above, it follows  
\[
	H^i(\overline{\mathcal{E}_{K3}},\mathbb{Q})\cong H^i(BO(\Lambda_{K3}),\mathbb{Q})\cong H^i(O(\Lambda_{K3});\mathbb{Q}).
\]
\end{proof}
Also note that as the group $O(\Lambda_{K3})$ is an arithmetic subgroup of $O(3,19)$, it has Kazhdan's Property (T) from which one also obtains that the first Betti number of $\overline{\mathcal{E}_{K3}}$ vanishes. But it would be interesting to know if there is anything which can be said about higher Betti numbers of the group $O(\Lambda_{K3})$, as this would give some information on the higher rational homotopy groups of $\overline{\mathcal{E}_{K3}}$. 
\subsubsection*{Acknoledgement}
The author acknowledges funding by the Deutsche Forschungsgemeinschaft (DFG, German Research Foundation) – 281869850 (RTG 2229).
\bibliography{Literatur}

\bibliographystyle{plainurl}

\end{document}